\documentclass[12pt]{article}

\usepackage{latexsym,amsthm}
\usepackage{amssymb}
\usepackage{amsfonts, amsmath}
\usepackage{amscd}
\usepackage{amstext}
\input{xypic.tex} \xyoption{all}
\usepackage[pdftex]{color, graphicx}

\usepackage{makeidx}
\usepackage{geometry}
\usepackage{indentfirst}
\makeindex

\usepackage[latin1]{inputenc}
\newcommand{\pn}{{\mathbb{P}^n}} \newcommand{\opn}{{\cal
 O}_{\mathbb{P}^n}}

\newtheorem{theorem}{Theorem}[section]

\newtheorem{proposition}[theorem]{Proposition}
\newtheorem{lemma}[theorem]{Lemma}
\newtheorem{corollary}[theorem]{Corollary}

\newtheorem{remark}[theorem]{Remark}

\newtheorem{definition}[theorem]{{\bf Definition}}

\begin{document}

\title{Pure resolutions of vector bundles on complex projective spaces}
\author{Marcos Jardim, Daniela Moura Prata}
\date{\today}
\maketitle

\begin{abstract} 
We prove three results on pure resolutions of vector bundles on projective spaces. First, we show that there are simple vector bundles of rank $n$ on $\pn$ with arbitrary homological dimension. We then analyze the pure resolutions given by the sheafification of the Koszul complex of a certain algebra and by the sheafification of the minimal free resolution of a compressed Gorenstein Artinian graded algebra, proving that their syzygies are simple vector bundles. Our main tool is a result originally established by Brambilla, for which we give an alternative proof using representations of quivers.
\end{abstract}

\section{Introduction}

Vector bundles over algebraic varieties play a central role in algebraic geometry. They have important applications in complex geometry, representation theory and mathematical physics.

Vector bundles over projective spaces are particularly important, and many interesting problems are still open. For instance, there are few known examples, in characteristic zero,  of non splitting vector bundles on $\pn$ of rank $r$, with
$2 \leq r \leq n-1$. They are: the Horrocks-Mumford bundle of rank $2$ on $\mathbb{P}^4$ \cite{HM}; Horrocks' {\it parent bundle} of rank $3$ on $\mathbb{P}^{5}$ \cite{Ho}; he instanton bundles of rank $2n$ on $\mathbb{P}^{2n +1}$, see for instance \cite{AO}; Sassakura's rank $3$ on $\mathbb{P}^4$ vector bundle \cite{Sas}; the weighted Tango bundles of rank $n-1$ on $\pn$ \cite{Cas}, that generalizes the Tango bundle \cite{T}; and more recently the rank $3$ vector bundles on $\mathbb{P}^4$, constructed by Kumar, Peterson e Rao \cite{KPR}.

A result of Bohnhorst and Spindler provides a link between low rank bundles and resolutions by sums of line bundles. More precisely, let $E$ be a vector bundle over $\mathbb{P}^n$. A {\it resolution} of $E$ is an exact sequence

\begin{eqnarray}\label{defresol}
\xymatrix{0 \ar[r]& F_d \ar[r] & F_{d-1} \ar[r] & \cdots  \ar[r] & F_1 \ar[r] & F_0 \ar[r] & E \ar[r]& 0}
\end{eqnarray}  where every $F_i$ splits as a direct sum of line bundles.

One can show that every vector bundle on $\pn$ admits resolution of the form (\ref{defresol}), see \cite[Proposition 5.3]{JM}.  The minimal number $d$ of such resolution is called {\it homological dimension} of $E$, and it is denoted by ${\rm hd}(E)$.  Bohnhorst and Spindler proved the following two results, \cite[Proposition 1.4]{BS} and \cite[Corollary 1.7]{BS}, respectively.

\begin{proposition}{\label{bsprop}}
Let $E$ be a vector bundle on $\pn$. Then
$$ {\rm hd}(E) \leq d \Longleftrightarrow H^{q}_{*}(E) = 0, \forall \; 1 \leq q \leq n-d-1.$$
\end{proposition}

\begin{proposition}
Let $E$ be a non splitting vector bundle on $\pn$. Then
$$ {\rm rk}(E) \geq n+1 - {\rm hd}(E).$$
\end{proposition}

This last Proposition tell us that in order to construct vector bundles of low rank, one must be able to study and construct vector bundles with high homological dimension; this is the main point of our paper. 

Consider the short exact sequences from the resolution (\ref{defresol}).
$$\xymatrix{0 \ar[r] & S_{i+1} \ar[r] & F_i \ar[r] & S_i \ar[r] & 0 }$$
for $i = 0, \cdots, d_1$, where $S_d = F_d $ and $S_0 = E$. The vector bundle $S_i$ is called the $i-$th {\it syzygy} of $E$. 

We say $E$ has a {\it pure resolution} of type $(d_0, \cdots, d_p) $, with $d_0 < d_1 < \cdots < d_p$,  if it is given by 
\begin{equation}\label{pure-res}
\xymatrix{
0 \ar[r] & \opn(-d_p)^{\beta_p} \ar[r] & \cdots \ar[r] & \opn(-d_1)^{\beta_1} \ar[r] & \opn(-d_0)^{\beta_0} \ar[r] & E \ar[r] & 0} \end{equation}
up to twist, where $\beta_i$ is the $i-$th total Betti number of $E, i=0, \cdots, p.$ 

This definition is analogous to the definition of pure resolutions of modules over a commutative ring. For references on this topic, see \cite{BoijSod}. Pure resolutions are being studied recently by many authors, see for instance Boij and Soderberg \cite{BoijSod}, Eisenbud, Floystad and Weyman \cite{Eins,Floy}, Eisenbud and Schreyer \cite{ES} and the references therein. Several conjectures about pure resolutions are proposed and proved in these articles. Such problems are, however, beyond the scope of this paper. 

The first main result of this paper, proved in Section \ref{simple}, states that one can construct simple vector bundles of rank $n$ on $\pn$ given by pure resolutions with any homological dimension.

\begin{theorem}\label{anyhd}
Let $n \in \mathbb{Z}$ be an integer with $n \geq 4$. For each $1 \leq l \leq n-1$, there exists a simple vector bundle $E$ on $\pn$, with rank $n$ and ${\rm hd}(E) = l$, given by resolution
$$\xymatrix{
0 \ar[r]& \opn(-d_0) \ar[r]^{\alpha_0} & \opn^{n+1} \ar[r]^>>>>{\alpha_1} & \opn(d_2)^{2n} \ar[r]^{\alpha_2} & \ldots \\
\ldots \ar[r]^<<<{\alpha_{d_{l-1}}} & \opn(d_l)^{2n} \ar[r]^>>>>{\alpha_l}& E  \ar[r] & 0
}$$
for integers $d_0, d_1, \cdots, d_l$ with  $d_0 > 0$, $d_1 = 0 < d_2 < \cdots < d_l$.
\end{theorem}

This Theorem is proved by induction on $l$ using a result originally established by Brambilla in \cite[Theorem 4.3]{Bb}; for the sake of completeness, we provide a new proof of Brambilla's theorem, albeit using different methods, namely a functorial correspondence between the so-called cokernel bundles and representations of a certain quiver, in Section \ref{quivers}.

The second goal of our paper, which also arises as an application of Brambilla's theorem, is trying to determine when the syzygies in a pure resolution of the form (\ref{pure-res}) are simple vector bundles. We prove two results in this direction.

First, let $R = k[x_0, \cdots, x_n]$ be the ring of polynomials in $n+1$ variables, and let $\{f_1, \cdots, f_{n+1}\}$ be a generic regular sequence of forms of degree $d$. Let $I = (f_1, \cdots, f_{n+1})$ be the ideal generated by the forms and $\pn = {\rm Proj}(R)$. Sheafifying the corresponding Koszul complex we have the following minimal free resolution.

\begin{eqnarray}\label{complexokosz}
\xymatrix{0 \ar[r] & \opn(-(n+1)d) \ar[r]^{\alpha^{t}} & \opn(-nd)^{n+1} \ar[r] & \ldots \ar[r]  & \opn(-d)^{n+1} \ar[r] & \opn \ar[r] & 0}
\end{eqnarray}
where $\alpha : \opn(-d)^{n+1} \rightarrow \opn$ is the map given by the forms. Our next result, proved in Section \ref{syzkoszul}, guarantees that its syzygies are simple vector bundles.

\begin{theorem}{\label{kos1}}
Consider the short exact sequences in complex (\ref{complexokosz}).

$$\xymatrix{ 0 \ar[r] & \opn(-(n+1)d) \ar[r] & \opn(-nd)^{\binom{n+1}{n}} \ar[r] & F_1 \ar[r] & 0 }$$

$$\xymatrix{0 \ar[r] & F_1 \ar[r] & \opn(-(n-1)d)^{\binom{n+1}{n-1}} \ar[r] & F_2 \ar[r] & 0 }$$

$$\vdots$$

$$\xymatrix{0 \ar[r] & F_{n-2} \ar[r] & \opn(-2d)^{\binom{n+1}{2}} \ar[r] & F_{n-1} \ar[r] & 0 }$$

$$\xymatrix{0 \ar[r] & F_{n-1} \ar[r] &  \opn(-d)^{n+1}  \ar[r] & \opn \ar[r] & 0}$$ 

Each $F_i$ is a simple vector bundle of rank ${\rm rk}(F_i) = \binom{n}{i}$, $1 \leq i \leq n-1$.
\end{theorem}

Now let $J = (f_1, \cdots, f_{\alpha_1})$ be an ideal of $R$ generated by $\alpha_1$ forms of degree $t+1$, such that $R/J$ is a compressed Gorenstein Artinian graded algebra of embedding dimension $n+1$ and socle degree $2t$. Using the Proposition \cite[Proposition 3.2]{MMRN}, sheafifying the minimal free resolution of $R/J$, we have

\begin{eqnarray}\label{compGorenstein}
\xymatrix{ 0 \ar[r] & \opn(-2t-n-1) \ar[r] & \opn(-t-n)^{\alpha_n} \ar[r] & \opn(-t-n+1)^{\alpha_{n-1}} \ar[r] & \cdots }
\end{eqnarray}
$$\xymatrix{
\cdots \ar[r] & \opn(-t-p)^{\alpha_p}  \ar[r] & \cdots \ar[r] &  \opn(-t-2)^{\alpha_2} \ar[r] & \opn(-t-1)^{\alpha_1} \ar[r]^<<<<{\beta} & \opn \ar[r] & 0
}$$ 
where $\beta$ is the map given by the $\alpha_1$ forms of degree $t+1$ and
$$\alpha_i = \binom{t+i - 1}{i-1} \binom{t+n+1}{n+1-i} - \binom{t+n-i}{n+1-i} \binom{t+n}{i-1}, \, for \, i=1, \cdots, n.$$

Finally, applying the same ideas we have the last main result of the paper, proved in Section \ref{syzcgag}.

\begin{theorem}\label{teoGor1}
Consider the short exact sequences on $(\ref{compGorenstein})$.
\begin{eqnarray}\label{seqsimples1}
\xymatrix{0 \ar[r] & \opn(-2t-n-1) \ar[r] & \opn(-t-n)^{\alpha_n} \ar[r] & F_1 \ar[r] & 0}
\end{eqnarray}
$$\xymatrix{0 \ar[r] & F_1 \ar[r] & \opn(-t-n+ 1)^{\alpha_{n-1}} \ar[r] & F_2 \ar[r] & 0}$$
 $$\xymatrix{& &  \vdots & &}$$
$$\xymatrix{0 \ar[r] & F_{n-p} \ar[r] & \opn(-t-p)^{\alpha_{p}} \ar[r] & F_{n-(p-1)} \ar[r] & 0}$$
$$\xymatrix{& & \vdots & & }$$
$$\xymatrix{0 \ar[r] & F_{n-1} \ar[r] & \opn(-t-1)^{\alpha_{1}} \ar[r] & \opn \ar[r] & 0}$$
Then for each $1 \leq i \leq n-1$, $F_i$ is simple.
\end{theorem}

\noindent{\bf Acknowledgments.} 
The first named author is partially supported by the CNPq grant number 302477/2010-1 and the FAPESP grant number 2011/01071-3. The second named author was supported by the FAPESP doctoral grant number 2007/07469-3, and is now supported by the FAPESP post-doctoral grant 2011/21398-7. We thank Rosa Maria Mir\'o-Roig and Laura Costa for their invaluable help. DP thanks the hospitality at the University of Barcelona.


\section{Representations of quivers and cokernel bundles}\label{quivers}

In this section we want to show how cokernel bundles can be related to representations of the so-called Kronecker quiver, and provide a new proof of Brambilla's result \cite{Bb}.

\subsection{Representations of quivers}
We begin by revising some basic facts about representations of quivers. Let $k$ be an algebraically closed field with characteristic zero.  A {\it quiver} $Q$ consists on a pair $(Q_0, Q_1)$ of sets where $Q_0$ is the set of vertices and $Q_1$ is the set of arrows and a pair of maps $t, h: Q_1 \rightarrow Q_0$ the tail and head maps. An example is the \emph{Kronecker quiver}, denoted $K_{n}$, which consists of $2$ vertices and $n$ arrows.

\begin{equation} \label{kronquiver}
\xymatrix{
\bullet \ar@<1.8ex>[r]^1 \ar@<-1.8ex>[r]^{\vdots}_n & \bullet
}
\end{equation}

A {\it representation} $A = (\{V_i\}, \{A_a\})$of $Q$ consists of a collection of finite dimensional $k-$vector spaces $\{V_i; i \in Q_0 \}$ together with a collection of linear maps $ \{A_a : V_{t(a)} \rightarrow V_{h(a)}; a \in Q_1\}$.  A {\it morphism} $f$ between two representations $A = (\{V_i\}, \{A_a\})$ and $ B = (\{W_i\}, \{B_a\})$ is a collection of linear maps $\{f_i\}$ such that for each $a \in Q_1$ the diagram bellow is commutative
$$\xymatrix{V_{t(a)} \ar[r]^{A_a} \ar[d]_{f_{t(a)}} & V_{h(a)} \ar[d]^{f_{h(a)}}\\
W_{t(a)} \ar[r]_{B_a} & W_{h(a)}}$$ 
With these definitions, representations of $Q$ form an abelian category hereby denoted by $Rep(Q)$. 

The \emph{Euler form} on $\mathbb{Z}^{Q_0}$ is a bilinear form associated to $Q$, given by
$$< \alpha, \beta>  = \sum_{i \in Q_0} \alpha_i \beta_i  - \sum_{a \in Q_1} \alpha_{t(a)} \beta_{h(a)}.$$ 
The \emph{Tits form} is the corresponding quadratic form, given by
$$q(\alpha) = <\alpha, \alpha>.$$
For instance, if $Q = K_n$ and $\alpha = (a,b) \in \mathbb{Z}^2$ the Tits form is $q(\alpha) = a^2 + b^2 - nab$.

To a given representation $A$ of $Q$ we associate a dimension vector $\alpha \in \mathbb{Z}^{Q_0}$ whose entries are $\alpha_i = \dim V_i$. We say that a dimension vector $\alpha \in \mathbb{Z}^{Q_0}$ is \emph{Schur root} if there exists a representation $A$ with dimension vector $\alpha$ such that ${\rm Hom}(A,A) = k$. Note that the condition ${\rm Hom}(A,A) = k$ is an open condition in the affine space of all representations with fixed dimension vector, thus if $\alpha$ is a Schur root, then ${\rm Hom}(A,A) = k$ for a generic representation with dimension vector $\alpha$. A reference for generic representations and Schur roots is \cite{S}. For more information about roots and root systems, we refer to $\cite{K}$. 

The following two facts will be very relevant in what follows. 

\begin{lemma}\label{Kac}
For any quiver $Q$, if $\alpha$ is a dimension vector satisfying $q(\alpha) > 1$, then every representation with dimension vector $\alpha$ is decomposable.
\end{lemma}

\begin{proposition}\label{schur} 
Let $Q$ be the Kronecker quiver with $n \geq 3$, and let $\alpha\in\mathbb{Z}^{2}$ be a dimension vector. If $q(\alpha) \leq 1 $, then $\alpha$ is a Schur root. In particular, if $q(\alpha)\leq 1$, then the generic representation of $Q$ with dimension vector $\alpha$ is indecomposable.
\end{proposition}

\subsection{Cokernel bundles}
The next definition is due to Brambilla in \cite{Bb}. Fix for simplicity $k = \mathbb{C}$. Let $E$ and $F$ be vector bundles on $\pn$, $n \geq 2$, satisfying the following conditions:

\begin{itemize}
\item[$(1)$] $E$ and $F$ are simple, that is, $\rm{Hom}(E,E) = \rm{Hom}(F,F) = \mathbb{C}$;
\item[$(2)$] ${\rm Hom}(F,E) =0;$
\item[$(3)$] ${\rm Ext}^{1}(F,E) = 0$;
\item[$(4)$] the sheaf $E^{*} \otimes F$ is globally generated;
\item[$(5)$] $W = {\rm Hom}(E, F)$ has dimension $w \geq 3$.
\end{itemize}

\begin{definition}
A \emph{cokernel bundle of type $(E,F)$} on $\pn$ is a vector bundle $C$ with resolution of the form
\begin{eqnarray}\label{cok}
\xymatrix{ 0  \ar[r] & E^{\oplus a} \ar[r]^{\alpha} & F^{\oplus b} \ar[r] & C \ar[r]& 0 }
\end{eqnarray}  
where $E, F$ satisfy the conditions (1) through (5) above, $a,b \in \mathbb{N}$ and $b \,{\rm{rk}} F - a \,{\rm rk} E \geq n$.
\end{definition}

To simplify the notation we will write $E^a$ instead of $E^{\oplus a}$ and so on. Cokernel bundles  of type $(E,F)$ form a full subcategory of the category of coherent sheaves on $\pn$; this category will be denoted by $\mathcal{C}(\mathbb{P}^n)$. Below we have some examples of cokernel bundles.

\begin{enumerate}
\item Bundle given by exact sequence
$$\xymatrix{0 \ar[r] & \opn(-d)^a \ar[r]^{\alpha} & \opn^b \ar[r] & C \ar[r] & 0}$$
with $d \geq 1$. The map $\alpha$ is given by a matrix of forms of degree $d$.

\item The bundle
$$\xymatrix{ 0 \ar[r] & \Omega^p(p)^c \ar[r]^{\alpha} & \opn^b \ar[r] & C \ar[r] & 0}$$
where $0 < p \leq n$ and $\alpha$ is a matrix given by $p-$forms.
\end{enumerate}

Let us now see how cokernel bundles are related to quivers. Fix a basis $\sigma = \{\sigma_1, \cdots, \sigma_w\}$ of ${\rm Hom}(E,F)$. 

\begin{definition}
A representation $A = (\{\mathbb{C}^a, \mathbb{C}^b\}, \{A_i\}_{i=1}^w)$ of $K_w$ is $(E,F, \sigma)-$\emph{globally injective} when the map 
$$ \alpha(P) := \sum_{i=1}^w A_i \otimes \sigma_i(P) ~:~ 
\mathbb{C}^a\otimes E_P \to \mathbb{C}^b\otimes F_P $$
is injective for every $P \in \mathbb{P}^n$; here, $E_P$ and $F_P$ denote the fibers of $E$ and $F$ over the point $P$, respectively. 
\end{definition}

$(E,F, \sigma)-$globally injective representations of $K_w$ form a full subcategory of the category of representations of $K_w$; we denote it by $Rep(K_w)^{gi}$. From now on, since $(E,F, \sigma)$ are fixed, we will just refer to globally injective representations. It is a simple exercise to establish the following properties of $Rep(K_w)^{gi}$.

\begin{lemma}\label{subob}
The category $Rep(K_w)^{gi}$ is closed under sub-objects, i.e. every subrepresentation $A'$ of a representation $A$ in $Rep(K_w)^{gi}$ is also in $Rep(K_w)^{gi}$.
\end{lemma}

\begin{lemma}
The category $Rep(K_{w})^{gi}$ is exact, i.e. $Rep(K_{w})^{gi}$ is closed under extentions and under direct summands.
\end{lemma}

Our next result relates the category of globally injective representations of $K_w$ to the category of cokernel bundles.

\begin{theorem}\label{Teo1}
For every choice of basis $\sigma$ of ${\rm Hom}(E,F)$, there is an equivalence between $Rep(K_w)^{gi}$, the category of $(E,F, \sigma)-$globally injective representations of $K_w$, and $\mathcal{C}(\pn)$ the category of cokernel bundles of type $(E,F)$.
\end{theorem}

\begin{proof}
Given a basis $\sigma$ of ${\rm Hom}(E,F)$, we construct a functor $L_\sigma : Rep(K_{w})^{gi} \rightarrow \mathcal{C}(\mathbb{P}^n)$ and show that it is essentially surjective and fully faithfull. 

Let $A = (\{\mathbb{C}^a, \mathbb{C}^b\}, \{A_i\}_{i=1}^w)$ be a globally injective representation of $K_{w}$. Define a map $\alpha : E^{a} \rightarrow F^{ b}$ given by 
$$ \alpha = A_1 \otimes  \sigma_1  + \cdots + A_w \otimes \sigma_w. $$
Since $A$ is globally injective, we have that $\dim {\rm coker}\, \alpha(P) = b\, {\rm rk}F - a \, {\rm rk} E$ for each $P \in \mathbb{P}^n$. Therefore $\alpha$ is injective as a map of sheaves, and $C:={\rm coker}\alpha$ is a cokernel bundle.

Now given two globally injective representations $A = (\{\mathbb{C}^a, \mathbb{C}^b\}, \{A_i\}_{i=1}^w)$ and $B = (\{\mathbb{C}^c, \mathbb{C}^d\}, \{B_i\}_{i=1}^w)$, and a morphism $f=(f_1, f_2)$ between them, let $L_\sigma(A) = C_1, L_\sigma(B) = C_2$ be the cokernel bundles and $\alpha_1$, $\alpha_2$ the maps associated to $A$ and $B$, respectively. We want to define a morphism $L_\sigma(f): C_1 \rightarrow C_2$.

Since we have $f_1: \mathbb{C}^a \rightarrow \mathbb{C}^c, f_2: \mathbb{C}^b \rightarrow \mathbb{C}^d$, we have maps $f_{1}^{'} = f_1 \otimes 1_E \in {\rm{Hom}}(E^a,E^c)$ and $f_{2}^{'} = f_2 \otimes  1_F \in {\rm Hom}(F^b, F^d)$. Consider the diagram
\begin{eqnarray}\label{diagg}
\xymatrix{ 0 \ar[r] & E^a \ar[d]_{f_{1}^{'}} \ar[r]^{\alpha_1} & F^b \ar[d]_{f_{2}^{'}} \ar[r]^{\pi_1}&  C_1 \ar[r] \ar@{-->}[d] & 0\\
0 \ar[r] & E^c \ar[r]^{\alpha_2} & F^d \ar[r]^{\pi_2} & C_2 \ar[r] & 0}
\end{eqnarray} 
where $\pi_1, \pi_2$ are the projections. Applying the left exact contravariant functor ${\rm Hom}( - , C_2)$ to the upper sequence on (\ref{diagg}) we find a map $\phi \in {\rm Hom}(C_1,C_2)$ and we define $L_\sigma(f) := \phi$.

Now given $C$ an object of $\mathcal{C}(\mathbb{P}^{n})$ we take $\alpha  = \sum_{i=1}^{w} A_i \otimes \sigma_i, \; {\rm with} \; A_i \in {\rm Hom}(\mathbb{C}^a, \mathbb{C}^b), i=1, \cdots, w.$ Then we have that $A = (\{\mathbb{C}^a, \mathbb{C}^b\}, \{A_i\}_{i=1}^w)$ is a globally injective representation of $Rep(K_{w})$ such that $F(A) = C$. Therefore $L_\sigma$ is essentially surjective.

Finally, we need to prove that $L_\sigma$ is fully faithful.  To check that it is full,  given
$\phi \in {\rm Hom}_{\mathcal{C}(\mathbb{P}^n)}(L_\sigma(A), L_\sigma(B))$ we want $f = (f_1,f_2) \in {\rm Hom}_{Rep(K_w)^{g.i}}(A,B)$ such that $F(f)= \phi$. Let $\tilde{\phi} = \phi \pi_1 \in {\rm Hom}(F^b, C_2)$. Let us apply the left exact covariant functor
${\rm Hom}(F^b, -)$ to the lower sequence on (\ref{diagg1}).

\begin{eqnarray}\label{diagg1}
\xymatrix{ 0 \ar[r] & E^a \ar@{-->}[d]_{f_{1}^{'}} \ar[r]^{\alpha_1} & F^b \ar@{-->}[d]_{f_{2}^{'}} \ar[r]^{\pi_1}&  C_1 \ar[r] \ar[d]^\phi & 0\\
0 \ar[r] & E^c \ar[r]^{\alpha_2} & F^d \ar[r]^{\pi_2} & C_2 \ar[r] & 0}
\end{eqnarray}
We have
$$\xymatrix{
0\ar[r] & {\rm Hom}(F^b, E^c) \ar[r]^{\rho_1}& {\rm Hom}(F^b,F^d) \ar[r]^{\rho_2} &  {\rm Hom}(F^b, C_2) \ar[r] & {\rm Ext}^{1}(F^b, E^c)}$$ 
since ${\rm Hom}(F^b, E^c) = {\rm Ext}^{1}(F^b, E^c) = 0$ then
\begin{eqnarray}\label{eqq}
\rho_2:  {\rm Hom}(F^b,F^d ) \rightarrow {\rm Hom}(F^b, C_2)
\end{eqnarray} 
is an isomorphism, so there is a morphism $f_{2}^{'}\in {\rm Hom}(F^b,F^d )$ such that
$$ \rho_2 (f^{'}_2) = \pi_2 f^{'}_2  = \phi \pi_1 $$
with $f^{'}_2 = f_2 \otimes 1_F$ and $f_2 \in {\rm Hom}(\mathbb{C}^b, \mathbb{C}^d)$.

Consider $\tilde{\tilde{\phi}} = f_{2}^{'} \alpha_1 \in {\rm Hom}(E^a, F^d)$. Applying the left exact covariant functor
${\rm Hom}(E^a, -)$ to the lower sequence on (\ref{diagg1}) we get
$$ \xymatrix{ 
0\ar[r] & {\rm Hom}(E^a, E^c) \ar[r]^{\gamma_1}& {\rm Hom}(E^a,F^d) \ar[r]^-{\gamma_2} &  {\rm Hom}(E^a, C_2) \ar[r] &\cdots
}$$

Once we have an exact sequence,
$$ \gamma_2(f^{'}_2\alpha_1) = \pi_2 f_2^{'}\alpha_1 = \phi \pi_1 \alpha_1 = 0 $$
then $f^{'}_2 \alpha_1 \in \ker{\gamma_2} = {\rm im}\, \gamma_1 $, and there is a map $f^{'}_1 \in {\rm Hom}(E^a, E^c)$ such that
$\gamma_1 (f_{1}^{'}) = \alpha_2 f^{'}_{1} = f^{'}_2 \alpha_1$ and $f^{'}_1 = f_1 \otimes 1_E$ with
$f_1 \in {\rm Hom}(\mathbb{C}^a, \mathbb{C}^c)$.

Since $\alpha_1 = \sum_{i=1}^{w} A_i \otimes \sigma_i$, $\alpha_2 = \sum_{i=1}^{w} B_i \otimes \sigma_i$,
$\alpha_2 f^{'}_1 = f^{'}_2 \alpha_1$, and $\sigma$ is a basis then $f_2 A_i = B_i f_1, i= 1, \cdots, w$, thus
$f = (f_1, f_2) \in {\rm Hom}_{Rep(K_w)^{g.i}}(A, B)$.

Now we need to prove that $L_\sigma(f) = \phi$. Suppose $L_\sigma(f) = \overline{\phi}$ such that
$\overline{\phi} \pi_1 = \pi_2 f^{'}_2 = \phi \pi_1$. Then $(\overline{\phi} - \phi) \pi_1 = 0$ and
$C_1 = {\rm im}\, \pi_1 \subset \ker(\overline{\phi} - \phi)$  therefore $\overline{\phi} = \phi$.

To prove that the functor is faithful, we must show that
$L_\sigma:{\rm Hom}_{Rep(K_w)^{g.i}}(A, B) \rightarrow {\rm Hom}_{\mathcal{C}(\mathbb{P}^n_{\mathbb{C}})}(L_\sigma(A), L_\sigma(B))$ is injective. Let $f = (f_1, f_2)$, $g=(g_1, g_2) \in {\rm Hom}(A, B)$ be morphisms such that $L_\sigma(f) = \phi_1 = \phi_2 = L_\sigma(g)$, that is, $\phi_1 - \phi_2 = 0$.

\begin{eqnarray}\label{diagg2}
\xymatrix{ 0 \ar[r] & E^a \ar[d]_{f_{1}^{'} - g^{'}_1} \ar[r]^{\alpha_1} & F^b \ar[d]_{f_{2}^{'} - g_{2}^{'}} \ar[r]^{\pi_1}&  C_1 \ar[r] \ar[d]^0 & 0\\
0 \ar[r] & E^c \ar[r]^{\alpha_2} & F^d \ar[r]^{\pi_2} & C_2 \ar[r] & 0}
\end{eqnarray}

Given $\phi_1 - \phi_2 = 0 \in {\rm Hom}(C_1, C_2)$, doing the same construction as before,
$$ 0 \pi_1 = 0 \in {\rm Hom}(F^b , C_2) \simeq {\rm Hom}(F^b,F^d )$$
with isomorphism given by $\rho_2$ in (\ref{eqq}). Since

$$\rho_2 (f^{'}_2 - g^{'}_2) = \pi_2 \circ (f^{'}_2 - g^{'}_2) = 0$$ then $f^{'}_2 - g^{'}_2 = 0$ and so $f^{'}_2 = g^{'}_2$. Similarly, $0 \alpha_1 = 0 \in {\rm Hom}(E^a, F^d)$ and

$$\gamma_1 (f^{'}_1 - g^{'}_1) = \alpha_2 (f^{'}_1 - g^{'}_1) = 0 \alpha_1 = 0.$$ Since $\gamma_1$ injective, $f^{'}_1 - g^{'}_1 = 0$, then $f^{'}_1 = g^{'}_1.$  Therefore $L_\sigma$ is faithful.

\end{proof}

\begin{remark} \rm
Note that the functor $L_\sigma$ depends on the choice of basis $\sigma$. Let $\sigma'$ be another basis for
${\rm Hom}(E,F)$. Let $L_{\sigma'}$ be the equivalence functor equivalence between the category of $(E,F,\sigma')$-globally injective representations of $K_w$ and the cokernel bundles on $\pn$. Then if $G$ is the inverse functor of $L_{\sigma'}$ we have that the functor $G\circ L_{\sigma'}$ gives an equivalence between the categories $(E,F,\sigma)$- and $(E,F,\sigma')-$globally injective representations of $K_w$.
\end{remark}

\begin{lemma}\label{decomp1}
For any choice of basis $\sigma$, the functor  $L_{\sigma}:Rep(K_{w})^{gi}\rightarrow\mathcal{C}(\mathbb{P}^n)$ defined above is additive and exact. In particular, if $R = R_1 \oplus R_2$ is a globally injective representation, then
$L_{\sigma}(R) \simeq L_{\sigma}(R_1) \oplus L_{\sigma}(R_2)$.
\end{lemma}

\begin{proof}
Checking the additivity of $L_{\sigma}$ is a simple exercise. We show its exactness in detail.

Let us prove that $L_{\sigma}$ preserves exact sequences. Let $R_1 = (\{\mathbb{C}^{a_1}, \mathbb{C}^{b_1}\}, \{A_i\})$, $R_2 = (\{\mathbb{C}^{a_2}, \mathbb{C}^{b_2}\}, \{B_i\})$ and $R_3 = (\{\mathbb{C}^{a_3}, \mathbb{C}^{b_3}\},\{C_i\} )$ be globally injective representations of $K_{w}$ and let $f: R_1 \rightarrow R_2$ and $g: R_2 \rightarrow R_3$  be morphisms such that the sequence 
$$\xymatrix{0 \ar[r] &R_1 \ar[r]^{f} & R_2 \ar[r]^{g} & R_3 \ar[r] & 0}$$
is exact. We want to prove that
$$\xymatrix{0 \ar[r] & C_1 \ar[r]^{\varphi} & C_2\ar[r]^{\psi} & C_3 \ar[r] & 0}$$
is also exact, where $C_i = L_{\sigma}(R_i), i=1,2,3$ and $\varphi = L_{\sigma}(f), \psi= L_{\sigma}(g)$. From the exact sequence of representations we get
$$\xymatrix{ & 0 \ar[d] & 0 \ar[d]& 0 \ar@{-->}[d] & \\
0 \ar[r] & E^{\oplus a_1} \ar[r]^{\alpha_1} \ar[d]^{1_{E}\otimes f_1} & F^{\oplus b_1}\ar[r]^{\pi_1} \ar[d]^{1_{F}\otimes f_2} & C_1 \ar[d]^{\varphi} \ar[r] & 0\\
0 \ar[r] & E^{\oplus a_2}\ar[r]^{\alpha_2}\ar[d]^{1_{E}\otimes g_1} & F^{\oplus b_2} \ar[r]^{\pi_2} \ar[d]^{1_{F}\otimes g_2} & C_2 \ar[r] \ar[d]^{\psi} & 0\\
0 \ar[r] & E^{\oplus a_3} \ar[r]^{\alpha_3} \ar[d] & F^{\oplus b_3} \ar[r]^{\pi_3} \ar[d] & C_3 \ar[r] \ar@{-->}[d] & 0\\
& 0 & 0 & 0 & \\}$$
We need to show that $\varphi$ is injective and $\psi$ is surjective.

\begin{itemize}

\item $\psi$ is surjective:

It follows from the fact that $\pi_3 (1_{F} \otimes g_2)$ is surjective.


\item $\varphi$ is injective.

Let us suppose $\varphi(s) = 0, s \in C_1$. Then $s = \pi_1(v), v \in F^{b_1}$ and
$$ 0 = \varphi \pi_1 (v) = \pi_2 (1_{F}\otimes f_2)(v).$$
Since $\ker \pi_2 = \rm im \alpha_2$, there is $u \in E^{ a_2}$ such that
\begin{equation}\label{Eq1}
(1_{F }\otimes f_2)(v) = \alpha_2 (u)
\end{equation}

\end{itemize}

Note that
$$ \alpha_3 ( 1_{E}\otimes g_1) (u) = (1_{F}\otimes g_2)(\alpha_2)(u) = (1_{F}\otimes g_2)(1_{F}\otimes f_2)(v) = 0 $$
and since $\alpha_3$ is injective, $(1_{E}\otimes g_1)(u) = 0$ so $u = (1_{E}\otimes f_1)(u')$ with $u' \in E^{a_1}$.
We have
$$ \alpha_2 (u) = \alpha_2 (1_{E}\otimes f_1)(u') = (1_{F}\otimes f_2) \alpha_1 (u'). $$
From $(\ref{Eq1})$ we have $(1_{F}\otimes f_2)(v) =(1_{F}\otimes f_2)( \alpha_1 (u'))$. Since $(1_{F}\otimes f_2)$ is injective, it follows that $v = \alpha_1(u')$ therefore
$$ s = \pi_1(v) = \pi_1 \alpha_1 (u') = 0.$$

Now suppose $R = R_1 \oplus R_2$. Let us prove that $L_{\sigma}(R_1 \oplus R_2) = L_{\sigma}(R_1) \oplus L_{\sigma}(R_2)$. We have the short exact sequence

\begin{eqnarray*} \xymatrix{
0 \ar[r] & R_1 \ar[r]^<<<<<{i_{R_1}} & R_1 \oplus R_2 \ar[r]^>>>>{ \pi_{R_2}} & R_2 \ar@<.3cm>[l]^>>>>{i_{R_2}} \ar[r] & 0  }
\end{eqnarray*}
where $i_{R_j}$ is the inclusion and $\pi_{R_j}$ the projection, $j=1,2$. Since the sequence above is split, $\pi_{R_2} \circ i_{R_2} = 1_{R_2}$. Now since $L_{\sigma}$ is an exact functor, we have
\begin{eqnarray}\label{seqsplit} \xymatrix{
0 \ar[r] & L_{\sigma}(R_1) \ar[r]^<<<<{L_{\sigma}(i_{R_1})} & L_{\sigma}(R_1 \oplus R_2) \ar[r]^>>>>{ L_{\sigma}(\pi_{R_2})} & L_{\sigma}(R_2) \ar[r] \ar@<.3cm>[l]^>>>>{L_{\sigma}(i_{R_2})} & 0  }
\end{eqnarray}
Then
$$ L_{\sigma}(\pi_{R_2} \circ i_{R_2}) = L_{\sigma}(\pi_{R_2}) \circ L_{\sigma}(i_{R_2}) = L_{\sigma}(1_{R_2}) = 1_{L_{\sigma}(R_2)} $$ therefore the sequence (\ref{seqsplit}) is split. Hence $L_{\sigma}(R_1 \oplus R_2) \simeq L_{\sigma}(R_1) \oplus L_{\sigma}(R_2)$.
\end{proof}

We are finally ready to establish the main result of this section, a new proof for a result due to Brambilla, cf. \cite[Theorem 4.3]{Bb}.

\begin{theorem}\label{decon}
Let $C$ be a cokernel bundle of type $(E,F)$, given by the resolution
\begin{eqnarray}\label{cok1}
 \xymatrix{ 0  \ar[r] & E^{ a} \ar[r]^{\alpha} & F^{ b} \ar[r] & C \ar[r]& 0 } ,\end{eqnarray}
and let $w=\dim{\rm Hom}(E,F)$.
\begin{itemize}
\item[(i)] If $C$ is simple, then $a^2 + b^2 - wab \leq 1$
\item[(ii)] If $a^2 + b^2 - wab \leq 1$, then there exists a non-empty open subset
$U\subset{\rm Hom}(E^a,F^b)$ such that for every $\alpha\in U$ the corresponding cokernel bundle is simple.
\end{itemize}
\end{theorem}

\begin{proof} 

To prove $(i)$, let $C$ be a cokernel bundle given by resolution (\ref{cok1}) and suppose $C$ is simple. By Theorem \ref{Teo1} there is a globally injective representation $R$ of $K_w$ such that $C = F(R)$. Since $F$ is full, $\mathbb{C} = {\rm Hom}(C,C) \simeq {\rm Hom}(R,R)$, thus  $R$ is simple and therefore, by Proposition \ref{Kac}, $q(a,b) = a^2+ b^2 -wab \leq 1.$

To prove $(ii)$,  if $a^2 + b^2 - wab \leq 1$, there is a generic representation $R$ with dimension vector $(a,b)$ such that $R$ is Schur,  by Proposition \ref{Kac}. Then there is a non-empty open subset $U \subset {\rm Hom(E^a, F^b)}$ such that  $\alpha \in U$ corresponds to $C = F(R)$. Since  ${\rm Hom}(C,C) \simeq {\rm Hom}(R,R) = \mathbb{C}$, it follows that $C$ is simple.

\end{proof}

The previous Theorem implies that if $q(a,b) > 1$ then $C$ is not simple. However, more is true, and it is not difficult to establish the following stronger statement.

\begin{proposition}\label{decon1}
Under the same conditions as in Theorem (\ref{decon}), if $a^2+ b^2 - w ab>1$, then $C$ is decomposable.
\end{proposition}
\begin{proof} 
Let $C$ be any cokernel bundle given by exact sequence (\ref{cok1}), such that $a^2 + b^2 - w ab  > 1$. Then there is a globally injective representation $R$ of $K_w$, such that $C = F(R)$ and $q(a,b)> 1$. By Lemma \ref{Kac}, $R$ is decomposable. Then by Lemma \ref{decomp1}, $C$ is decomposable.
\end{proof}

Under more restrictive conditions, Brambilla proved in \cite[Theorem 6.3]{Bb} that if $C$ is a \emph{generic} cokernel bundle such that $a^2+ b^2 - w ab \geq 1$ then $C \simeq C^n_k \oplus C^m_{k+1}$, where $C_k$ and $C_{k+1}$ are {\it Fibonacci bundles}, $n,m \in \mathbb{N}$ (we refer to \cite{Bb} for the definition of Fibonacci bundles).

Recall that A vector bundle $E$ on $\mathbb{P}^n$ is exceptional if it is simple and ${\rm Ext}^p(E,E) = 0$ for $p \geq 1$.

\begin{remark}\label{rmk}\rm
Since the path algebra associated to a Kronecker quiver $Q$ is hereditary, then ${\rm Ext}^p(R,R) = 0$ for every representation $R$ of
$Q$ and $ p \geq 2$, see \cite{Gab} for references. Since the functor $L_{\sigma}$ is exact, we have
${\rm Ext}^1(R,R) \simeq {\rm Ext}^1(L_{\sigma}(R), L_{\sigma}(R))$. Now we know from \cite{S} that
$$ q(a,b) = \dim {\rm Hom}(R,R) - \dim {\rm Ext}^1(R,R) $$
hence if $C$ is exceptional then ${\rm Ext}^1(R,R) = 0$ and $q(a,b)=1$.
\end{remark}

However the converse of the previous observation is not true. For instance, consider the generic cokernel bundle given by exact sequence
$$ \xymatrix{ 0 \ar[r] & \mathcal{O}_{\mathbb{P}^3} \ar[r] & \mathcal{O}_{\mathbb{P}^3}(4)^{35} \ar[r] & C \ar[r] & 0 }. $$
We have $q(1,35)=1 + {35}^2 - 35.1.35 = 1$, but from the long exact sequence of cohomologies, ${\rm Ext}^2(C,C) \simeq \mathbb{C}^{35}$ hence $C$ is not exceptional.

We conclude this section with a Lemma that will be useful later on.

\begin{lemma}\label{simpquiv}
Let $C_1$ and $C_2$ be cokernel bundles given by exact sequences
$$ \xymatrix{
0 \ar[r] & E_1 \ar[r]^{\alpha_1} & \opn^{b_1} \ar[r] & C_1 \ar[r] & 0}$$ and $$\xymatrix{0 \ar[r] & \opn^{b_2} \ar[r]^{\alpha_2} &
E_2 \ar[r] & C_2 \ar[r] & 0
}.$$ 

\begin{itemize}
\item[$(a)$] If $h^0(E_1^*) \geq b_1$ then there exists a non-empty open subset  $U_1 \subset {\rm Hom(E_1, \opn^{b_1})}$ such that for every $\alpha_1 \in U_1$ the corresponding bundle $C_1$ is simple;
\item[$(b)$]  If $h^0(E_2) \geq b_2$ then there exists a non-empty open subset  $U_2 \subset {\rm Hom( \opn^{b_2}, E_2)}$ such that for every $\alpha_2 \in U_2$ the corresponding bundle $C_2$ is simple.
\end{itemize}

\end{lemma}

\begin{proof}
If $h^0(E_1^*) \geq b_1$, we have $1 + {b_1}^2 - h^0(E_1^*) b_1 \leq 1$, therefore by Theorem \ref{decon} $(ii)$ it follows that $C_1$ is simple. The proof of item (b) is similar.
\end{proof}

\subsection{Steiner bundles}

As an important example of cokernel bundles we have the \emph{Steiner bundles} introduced by Mir\'o-Roig and Soares in \cite{SM}. These bundles generalize the definition of Steiner bundles in the sense of Dolgachev and Kapranov, cf. \cite{DK}. The definition is the following.

Let $X$ be a smooth irreducible algebraic variety $X$ over an algebraically closed field of characteristic zero. Let
$\mathcal{D} = D^b(\mathcal{O}_X - mod)$ be the bounded derived category of the abelian category of coherent sheaves of
$\mathcal{O}_X -$modules on $X$. A vector bundle $E$ on $X$ is called a \emph{Steiner bundle of type} $(E,F)$ if it is defined by an exact sequence of the form
\begin{eqnarray}\label{steiner}
\xymatrix{ 0 \ar[r] & E^a \ar[r]^{\alpha} & F^b \ar[r] & E \ar[r] & 0 } 
\end{eqnarray}
where $a,b \geq 1$ and $(E,F)$ is an ordered pair of vector bundles on $X$ satisfying the following two conditions:
\begin{itemize}
\item[$(i)$] $(E,F)$ is {\it strongly exceptional}; that is, $E, F$ are exceptional,
 $${\rm Ext}^p (F, E) = 0, \;  \forall \; p \in \mathbb{Z} \;  \mbox{and} \;  {\rm Ext}^p(E,F) = 0, \; p \neq 0.$$
\item[$(ii)$] $E^{*} \otimes F$ is globally generated.
\end{itemize}

Note that the Steiner bundles of type $(E, F)$ where $\dim {\rm Hom}(E,F) \geq 3$, are a particular case of cokernel bundles.

\vspace{.3cm}

\begin{proposition}
Let $C_1, C_2$ be Steiner bundles of type $(E, F)$. Then ${\rm Ext}^p(C_1, C_2) = 0$ for $p \geq 2$.
\end{proposition}

\begin{proof}
Suppose $C_1$ and $C_2$ are Steiner bundles given by short exact sequences
\begin{eqnarray}\label{eqsteiner1}
\xymatrix{0 \ar[r] & E^{a_1} \ar[r] & F^{b_1} \ar[r] & C_1 \ar[r] & 0}
\end{eqnarray}
and
\begin{eqnarray}\label{eqsteiner2}
\xymatrix{0 \ar[r] & E^{a_2} \ar[r] & F^{b_2} \ar[r] & C_2 \ar[r] & 0}
\end{eqnarray}

Applying the functor ${\rm Hom}( - ,F)$ to the sequence ($\ref{eqsteiner1}$) we have ${\rm Ext}^p(C_1, F) = 0$, $p \geq 2$. Applying ${\rm Hom}( - , E)$ to the same sequence, we obtain ${\rm Ext}^q(C_1, E) = 0, q \geq 0$. Filnally applying the functor ${\rm Hom}(C_1, - )$
to the sequence ($\ref{eqsteiner2}$) we conclude that ${\rm Ext}^j(C_1, C_2) = 0$ for $j \geq 2.$
\end{proof}

As a corollary we have the following. Recall that $w = \dim~{\rm Hom}(E, F).$

\begin{corollary}
Let $C$ be a Steiner bundle of type $(E, F)$ given by short exact sequence.
$$ \xymatrix{0 \ar[r] & E^a \ar[r]^{\alpha} & F^{b} \ar[r] & C \ar[r] & 0} $$
Then
\begin{itemize}
\item[$(i)$] If $C$ is exceptional then $q(a,b) = a^2 + b^2 - w a b =1.$
\item[$(ii)$] If $q(a,b) = 1 $ then there is a non-empty open subset $U \subset {\rm Hom}(E^a, F^b)$ such that for every $\alpha \in U$ the corresponding bundle $C$ is exceptional.
\end{itemize}
\end{corollary}

\begin{proof} 
\begin{itemize}
\item[$(i)$] This follows from Remark \ref{rmk}.
\item[$(ii)$] Suppose $q(a,b) = 1$. By Theorem \ref{decon} item $(ii)$ there exists a non-empty open subset
$U \subset {\rm Hom}(E^a, F^b)$ such that for every $\alpha  \in U$ the associated bundle $C$ is simple. From Remark \ref{rmk} we see that ${\rm Ex}t^1(C,C) = 0$. Finaly, using the previous Proposition, we have ${\rm Ext}^{p}(C,C) = 0$ for $p \geq 2$. Hence $C$ is exceptional.
\end{itemize}
\end{proof}

\begin{remark}\rm
Soares also proved in \cite[Theorem 2.2.7]{Soa}, using a different method, that a generic Steiner bundle of type $(E,F)$ given by short exact sequence $(\ref{steiner})$ is exceptional if and only if $q(a,b) = 1$.
\end{remark}


\section{Simple vector bundles on $\pn$ with arbitrary homological dimension}\label{simple}

In this section we construct simple vector bundles of rank $n$ on $\pn$ with arbitrary homological dimension and with pure resolution. More precisely, our goal is to prove the following result.

\begin{theorem}{\label{anylhd}}
Given integers $n \geq 4$ and $1 \leq l \leq n-1$, there exists a simple vector bundle $E$ over $\pn$ with $\rm{rk}(E) = n$ of rank $n$ given by a pure resolution of the form
\begin{equation}{\label{longseq}}
\xymatrix{ 0 \ar[r]& \opn(-d_0) \ar[r]^{\alpha_0} & \opn^{n+1} \ar[r]^>>>>{\alpha_1} & \opn(d_2)^{2n} \ar[r]^{\alpha_2} & \ldots \\
\ldots \ar[r]^<<<{\alpha_{{l-1}}} & \opn(d_l)^{2n} \ar[r]^>>>>{\alpha_l}& E  \ar[r] & 0 }
\end{equation}
where the integers $d_0, d_2, \cdots, d_l$ satisfy $d_0 > 0$ and $0 < d_2 < \cdots < d_l$.
\end{theorem}

\begin{proof}
We proceed by induction on $l$. First, for $l=1$, consider the vector bundle $E_1$ given by exact sequence
$$ \xymatrix{0 \ar[r] & \opn(-d_0) \ar[r]^{\alpha_0} & \opn^{n+1} \ar[r] & E_1 \ar[r] & 0 }$$ where $\alpha_0$ is a generic map given by $n+1$ forms of degree $d_0$. Clearly $E_1$ has $\rm{hd}(E_1) = 1$, and it is stable by \cite[Theorem 2.7]{BS}.

Now let $E_{t-1}$ be a simple vector bundle of rank $n$ and homological dimension $t-1$ admitting a resolution of the form
\begin{eqnarray}\label{seqex}
0\to \opn(- d_0) \stackrel{\alpha_0}{\longrightarrow} \opn^{n+1} \stackrel{\alpha_1}{\longrightarrow}
\cdots \stackrel{\alpha_{t-2}}{\longrightarrow} \opn(d_{t-1})^{2n} \stackrel{\alpha_{t-2}}{\longrightarrow}
E_{t-1} \to 0
\end{eqnarray}
Choose $d_t > \max\{d_{t-1}, D_t\}$ where $E^{*}_{t-1}$ is $D_t-$regular and such that $h^{0}(E^{*}(d_t)) \geq 2n$. The degeneracy locus of a generic map $\alpha : E_{t-1} \rightarrow \opn(d_t)^{2n}$, i.e. is the variety
$$\Delta_{\alpha} = \{ P \in \pn  \mid \alpha_P: {E_{t-1}}_P \rightarrow \opn(d_t)^{ 2n}_P \; \mbox{is not injective } \}.$$ 
has codimension $2n-n+1=n+1$, see \cite[Chapter II]{ACGH}; hence $\Delta_{\alpha}$ is empty, and $E_t =$coker$\alpha$ is a vector bundle. We argue that it is in fact a cokernel bundle:

\begin{itemize}
\item[$(1)$] $E_{t-1}$ and $\opn(d_t)$ are simple;
\item[$(2)$] We must check that \rm{Hom}$(\opn(d_t), E_{t-1}) \simeq H^{0}(E_{t-1}(-d_t)) = 0$. Breaking (\ref{seqex}) into exact sequences, twisting by $\opn(-d_t)$, and passing to the long exact sequence of cohomologies we have
$$ 0 \to H^{0}(E_{t-2}(-d_t)) \to H^{0}(\opn(d_{t-1}-d_t)^{2n}) \to H^{0}(E_{t-1}(-d_t)) \to H^{1}(E_{t-2}(-d_t)) \to 0. $$
Since $H^{0}(\opn(d_{t-1}-d_t)=0$, it follows that $H^{0}(E_{t-1}(-d_t)) \simeq H^{1}(E_{t-2}(-d_t))$. By Proposition \ref{bsprop}, since hd$(E_{t-2}) =t-2 \leq  n-2$ we have $H^{1}(E_{t-2}(-d_t)) = 0.$
\item[$(3)$] Next, we must check that $\rm{Ext}^{1}(\opn(d_t),E_{t-1}) \simeq {\rm{ H}}^1(\pn, E_{t-1}(-d_t))=0$. It again follows from Proposition \ref{bsprop}, since hd$(E_{t-1}) = t-1 \leq n-2$.
\item[$(4)$] $E_{t-1}^{*}\otimes \opn(d_t)$ is globally generated, by our choice of $d_t$.
\item[$(5)$] $\dim{\rm{ Hom}}(E_{t-1}, \opn(d_t)) = h^{0}(E_{t-1}^{*}(d_t)) \geq 2n$, also by our choice of $d_t$.
\end{itemize} 

Therefore by definition, $E_t$ is a cokernel bundle, as desired; since
$$ q(1,2n) = 1 + 4n^2 - h^{0}(E^{*}(d_t)) 2n \leq 1 , $$
by Theorem \ref{decon}, we conclude that $E_t$ is simple.

It also follows that $E_t$ has a pure resolution of the form
\begin{eqnarray}\label{resEt}
\xymatrix{0 \ar[r] & \opn(-d_0) \ar[r]^{\alpha_0} & \opn^{n+1} \ar[r]^{\alpha_1} & \opn^{2n}(d_2) \ar[r]^{\alpha_2} & \cdots \\
 \ar[r]^<<<<{\alpha_{t-2}} & \opn(d_{t-1})^{2n}\ar[r]^{\alpha_{t-1}} & \opn(d_t)^{2n} \ar[r] & E_t \ar[r] & 0  }.
 \end{eqnarray}
We have to check that  hd$(E_t)\leq t$. Indeed, first note that, by  resolution (\ref{resEt}), hd$(E_t)\leq t$. Suppose hd$(E_t) \leq t-1$. Using Proposition \ref{bsprop} we have
$$\rm{H}^{q}(E_t(l))=0, \; 1 \leq q \leq n - (t-1) -1 , \; \forall l \in \mathbb{Z}. $$
Consider the short exact sequences of (\ref{resEt}). We have the sequence
$$ \xymatrix{0 \ar[r] & E_{t-1} \ar[r] &  \opn(d_t)^{2n} \ar[r] & E_t \ar[r] & 0}. $$
Since $2 \leq t \leq n-1$, applying the long exact sequence of cohomologies we have
$$ \xymatrix{0 \ar[r] & H^{n-t}(E_t(l)) \ar[r] & H^{n-t+1}(E_{t-1}(l)) \ar[r] & 0}\; \forall l \in \mathbb{Z} $$
thus $H^{n-t}(E_t(l)) \simeq H^{n-t+1}(E_{t-1}(l))$. Since hd$(E_{t-1}) = t-1$, it follows Proposition \ref{bsprop},
$$ \exists \; l_0 \in \mathbb{Z} ;\; H^{n-t+1}(E_{t-1}(l_0)) \neq 0,$$
hence $H^{n-t}(E_{t}(l_0))\neq 0 $, which contradicts hd$(E_t) \leq t-1$. Therefore, we must have hd$(E_t) = t$.

Using Proposition \ref{bsprop} once again, we can check that hd$(E_t) = t$, thus completing the proof.
\end{proof}


\section{Simplicity of syzygies of the Koszul complex}\label{syzkoszul}

From now on let $R = k[x_0, \ldots,x_n]$ be the ring of polynomials in $n+1$ variables with coefficients on an algebraically closed field $k$ of characteristic zero. Let $\{f_1, \ldots, f_{n+1}\}$ be a generic regular sequence of forms of degree $d$, and let
$I = (f_1, \ldots, f_{n+1})$ be the ideal generated by these forms and $\pn = {\rm Proj}(R)$. The \emph{Koszul complex}
$K(f_1, \cdots,f_{n+1})$ is given by
\begin{equation}\label{Kf}
\xymatrix{0 \ar[r] & \bigwedge^{n+1}T \ar[r] & \bigwedge^{n}T \ar[r] & \cdots \ar[r] & \bigwedge^{1} T \ar[r] & R \ar[r] & R/I \ar[r] & 0 }
\end{equation}
where $T = R(-d)^{n+1}$. 

Sheafifying the Koszul complex (\ref{Kf}) we get the follwoing exact sequence of vector bundles:
\begin{eqnarray}\label{Kcomp}
\xymatrix{0 \ar[r] & \opn(-(n+1)d) \ar[r]^{\alpha^{t}} & \opn(-nd)^{n+1} \ar[r] & \ldots \ar[r]  & \opn(-d)^{n+1} \ar[r] & \opn \ar[r] & 0}
\end{eqnarray} where $\alpha : \opn(-d)^{n+1} \rightarrow \opn$ is the map given by the forms $\{f_1, \ldots, f_{n+1}\}$. Breaking the above complex (\ref{Kcomp}) into short exact sequences, we obtain
$$\xymatrix{ 0 \ar[r] & \opn(-(n+1)d) \ar[r] & \opn(-nd)^{\binom{n+1}{n}} \ar[r] & F_1 \ar[r] & 0 }$$
$$\xymatrix{0 \ar[r] & F_1 \ar[r] & \opn(-(n-1)d)^{\binom{n+1}{n-1}} \ar[r] & F_2 \ar[r] & 0 }$$
\begin{equation}\label{syzygies} \vdots \end{equation}
$$\xymatrix{0 \ar[r] & F_{n-2} \ar[r] & \opn(-2d)^{\binom{n+1}{2}} \ar[r] & F_{n-1} \ar[r] & 0 }$$
$$\xymatrix{0 \ar[r] & F_{n-1} \ar[r] &  \opn(-d)^{n+1}  \ar[r] & \opn \ar[r] & 0}$$
It is not difficult to see that the syzygy sheaves $F_i$, $1 \leq i \leq n-1$, are actually locally free. Morever, it follows from \cite[Proposition 1.1.27]{MR}, that the above complex (\ref{Kcomp}) is \emph{self dual up to twist}, i.e. the complex and its dual are isomorphic up to twist.

The main goal of this section is to prove the following Theorem, again as an application of Theorem \ref{decon}.

\begin{theorem}\label{Kos}
The syzygies $F_i$, $1 \leq i \leq n-1$, of the Koszul complex (\ref{Kcomp}) are simple vector bundles of rank ${\rm rk}(F_i) = \binom{n}{i}$ and ${\rm hd}(F_i) = i$
\end{theorem}

\begin{proof} Again, we proceed by induction on $i$. We start by showing that ${\rm hd}(F_i) = i$. Indeed, it is clear that ${\rm hd}(F_1) = 1$. Assume that ${\rm hd}(F_j) = j$ for $2 \leq j \leq t-1$, $t \leq n-2$. Now $F_t$ and $F_{t-1}$ fit into the short exact sequence 
$$ \xymatrix{ 0 \ar[r] & F_{t-1} \ar[r] & \opn(-(n-t+1)d)^{\binom{n+1}{n-t+1}} \ar[r] & F_{t} \ar[r] & 0 }. $$
Therefore, from the long exact sequence of cohomologies
$$ \xymatrix{ \cdots \ar[r]  & H^i(F_{t}) \ar[r] & H^{i+1}(F_{t-1}) \ar[r] & H^{i+1}(\opn(-(n-t+1)d))^{\binom{n+1}{n-t+1}} \ar[r] & \cdots} $$
we get
$$ H^i(F_{t}) \simeq  H^{i+1}(F_{t-1}),\; \text{for}\; 1 \leq i \leq n-2.$$
Since ${\rm hd}(F_{t-1}) = t-1$,  we know from Proposition \ref{bsprop},  that $H^p_{*}(F_{t-1}) = 0$ for $1 \leq p \leq n-t$. Therefore $H^p_{*}(F_{t}) = 0$ for $1 \leq p \leq n-t-1$, hence ${\rm hd}(F_{t}) \leq t$ by Proposition \ref{bsprop}. 

Now suppose ${\rm hd}(F_t) \leq t-1$, then $H^p_{*}(F_t) = 0$ for $1 \leq p \leq n-t$. Since $H^{n-t}(F_t) \simeq H^{n-t+1}(F_{t-1})$ and ${\rm hd}(F_{t-1}) = t-1$, there must be $l_0 \in \mathbb{Z}$ such that $H^{n-t}(F_t(l_0)) \simeq H^{n-t+1}(F_{t-1}(l_0)) \neq 0 $. Hence ${\rm hd(F_t)} = t$.

Let us now prove the simplicity of each $F_i$; since the Koszul complex is self-dual up to twist, is sufficient to prove that $F_i$ is simple for $1 \leq i \leq \frac{n-1}{2}$. Again, we argue by induction on $i$. For $i = 1$, the first sequence in (\ref{syzygies}) and \cite[Theorem 2.7]{BS} imply that $F_1$ is stable, and therefore simple.

Now suppose that the statement is true for $2 \leq k \leq i-1$. Consider the Koszul complex twisted by $\opn((n-i+1)d)$, and let $\overline{F}_{i-1} = F_{i-1}((n-i+1)d)$ and let $\alpha_i : \overline{F}_{i-1} \rightarrow \opn^{\binom{n+1}{i}}$ be a generic map. Take $\overline{F}_i = {\rm coker}\alpha_{i-1}$; we check that it is a cokernel bundle. Clearly $\opn$ and $\overline{F}_{i-1}$ (by the induction hypothesis) are simple, and it is easy to check that ${\rm Hom}(\opn,\overline{F}_{i-1})=0$, since
$$ H^0(\pn,\overline{F}_{i-1}) \simeq H^j(\pn, \overline{F}_{i-(j+1)}), \; 0 \leq j \leq i-2 $$
and $H^{i-2}(\pn,\overline{F}_{1}) = 0$

Next, we check that ${\rm Ext}^1(\opn, \overline{F}_{i-1}) = 0$. Since ${\rm hd}(\overline{F}_{i-1}) = i-1$, hence, by Proposition \ref{bsprop}, it follows that 
$${\rm Ext}^{1}(\opn, \overline{F}_{i-1}) = H^1(\pn, \overline{F}_{i-1}) = 0.$$

To see that $\overline{F}_{i-1}^{*}$ is globally generated, take the sequence
$$ \xymatrix{
0 \ar[r] & \overline{F}_{i}^{*} \ar[r] & \opn^{\binom{n+1}{i}} \ar[r]^{\alpha_{i-1}^{*}} & \overline{F}_{i-1}^{*} \ar[r] & 0
}. $$
It follows that $\overline{F}_{i-1}^{*}$ is globally generated, and by the Koszul complex,
$h^0(\overline{F}_{i-1}^{*}) = \binom{n+1}{i}$. Then Lemma \ref{simpquiv} implies that $\overline{F}_{i}$ is a simple cokernel bundle, and from the Koszul complex $\overline{F}_{i}^{*}$ and $F_{n-1}(id)$ are $i$-th syzygyes of the ideal $I$. Therefore
$$ \overline{F}_i^{*} \simeq F_{n-i}(id) \simeq F_i((n-i+1)d). $$ 
and $F_i$ is simple, as desired.

Regarding the claim on the rank of $F_i$, note from the short exact sequences of Koszul complex (\ref{syzygies}) that
$${\rm rk}(F_{n-k}) = \binom{n+1}{k} - {\rm rk}(F_{n-(k-1)}), \;\; 2 \leq k \leq n.$$
For the case $i=1$ we have the resolution
$$ \xymatrix{0 \ar[r] & F_{n-1} \ar[r] & \opn(-d)^{n+1} \ar[r] & \opn \ar[r] & 0} $$
and ${\rm rk}F_{n-1} = n = \binom{n}{1}$. Assuming that ${\rm rk}F_{n-(i-1)}= \binom{n}{i-1}$, we have that
$$ {\rm rk}F_{n-i} = \binom{n+1}{i}- {\rm rk}F_{n-(i-1)} = \binom{n+1}{i} - \binom{n}{i-1} = \binom{n}{i}. $$
Since $F_i$ is isomorphic to $F_{n-i}^{*}$ up to a twist, follows that ${\rm rk}F_i = \binom{n}{i}$ for $1 \leq i \leq n-1$.
\end{proof}


\section{Compressed Gorenstein Artinian Graded Algebras}\label{syzcgag}

In this section we apply the same ideas of previous Section to a different minimal free resolution and find other examples of simple vector bundles on $\pn$. For more details on the definitions used below, we refer to \cite{iarr,MMRN}.

Let $R$ be a local ring with maximal ideal $M$ over a field $k$. Let $I \subset R$ be an ideal such that $A = R/I$ is Artinian with maximal ideal $ m = M+I$. The {\it socle} of $A$, denoted ${\rm Soc} A$,  is the annihilator $(0 : m) \subset A$.

The Artinian algebra $A$ is {\it Gorenstein} (or $I$ is Gorenstein) if and only if $(0: m)$ has length one. In this case, the largest integer $j$ such that $m^j \neq 0$ is the {\it socle-degree} of $A$.

Now let $k$ be a field and $R = k[x_0, \cdots, x_n]$ be the ring of polynomials. An Artinian $k-$algebra $A = R/I$ has
{\it socle degrees} $(s_1, \cdots, s_t)$, if the minimal generators of its socle (as $R-$module) have degrees $s_1 \leq \cdots \leq s_t$.

 The Hilbert-function of $A$ is denoted by $h_A(t) := \dim_k A_t. $ If the Artinian algebra $A$ has socle degrees $(s_1, \cdots, s_t)$, then $s = {\rm max }\{s_1, \cdots, s_t\}$ is called the {\it socle degree} of $A$. For fixed socle degrees, a graded Artinian algebra is {\it compressed}, if it has maximal Hilbert function among all graded Artinian algebras with that socle degrees.

Let $\pn = {\rm Proj}(R)$. Let $I = (f_1, \ldots, f_{\alpha_1}) $ be an ideal generated by $\alpha_1$ forms of degree $t+1$, such that the algebra $A = R/I$ is a compressed Gorenstein Artinian graded algebra of embedding dimension $n+1$ and socle degree $2t$. Thus, by Proposition $3.2$ of \cite{MMRN}, the minimal free resolution of $A$ is

\vspace{.3cm}

$$ \xymatrix{
0 \ar[r] & R(-2t-n-1) \ar[r] & R(-t-n)^{\alpha_n} \ar[r] & \cdots \ar[r] & R(-t-p)^{\alpha_p} \ar[r] & \cdots & } $$
$$ \xymatrix{
\cdots \ar[r] & R(-t-3)^{\alpha_3} \ar[r] &  R(-t-2)^{\alpha_2} \ar[r] & R(-t-1)^{\alpha_1} \ar[r] & R \ar[r] & A \ar[r] & 0
}$$ 
where
$$ \alpha_i = \binom{t+i - 1}{i-1} \binom{t+n+1}{n+1-i} - \binom{t+n-i}{n+1-i} \binom{t+n}{i-1}, \, for  \, i=1, \cdots, n. $$

Sheafifying the complex above we have

\begin{eqnarray}\label{compG}
\xymatrix{ 0 \ar[r] & \opn(-2t-n-1) \ar[r] & \opn(-t-n)^{\alpha_n} \ar[r] & \opn(-t-n+1)^{\alpha_{n-1}} \ar[r] & \cdots }
\end{eqnarray}
$$\xymatrix{ \cdots \ar[r] & \opn(-t-p)^{\alpha_p}  \ar[r] & \cdots \ar[r] &  \opn(-t-2)^{\alpha_2} \ar[r] & \opn(-t-1)^{\alpha_1} \ar[r]^<<<<{\beta} & \opn \ar[r] & 0 }$$  where $\beta$ is the map given by the $\alpha_1$ forms of degree $t+1$.

Consider the exact sequences

\begin{eqnarray}\label{seqsimp}
\xymatrix{0 \ar[r] & \opn(-2t-n-1) \ar[r] & \opn(-t-n)^{\alpha_n} \ar[r] & F_1 \ar[r] & 0}
\end{eqnarray}
$$\xymatrix{0 \ar[r] & F_1 \ar[r] & \opn(-t-n+ 1)^{\alpha_{n-1}} \ar[r] & F_2 \ar[r] & 0}$$
 $$\xymatrix{& &  \vdots & &}$$
$$\xymatrix{0 \ar[r] & F_{n-p} \ar[r] & \opn(-t-p)^{\alpha_{p}} \ar[r] & F_{n-(p-1)} \ar[r] & 0}$$
$$\xymatrix{& & \vdots & & }$$
$$\xymatrix{0 \ar[r] & F_{n-1} \ar[r] & \opn(-t-1)^{\alpha_{1}} \ar[r] & \opn \ar[r] & 0}$$

\vspace{.2cm}

\begin{lemma}\label{dimhomcg}
For each $1 \leq i \leq n-1$, we have ${\rm hd}(F_i) = i$.
\end{lemma}
\begin{proof} Let us prove it by induction on $i$. The case $i=1$ is not difficult to check.

So supose that for $2 \leq j \leq l-1$, ${\rm hd}(E_j) = j$. Consider the short exact sequence
$$ \xymatrix{0 \ar[r] & F_{l-1} \ar[r] & \opn(-t-n+l-1)^{\alpha_{n-l+1}} \ar[r] & F_l \ar[r] & 0 } $$
then by the long exact sequence of cohomologies
$$ \xymatrix{ \cdots \ar[r] & H^{k}(\opn(-t-n+l-1))^{\alpha_{n-l+1}} \ar[r] & H^{k}(F_l) \ar[r] & H^{k+1}(F_{l-1}) \ar[r] & \cdots } $$ we have
$$ H^{k}(F_l) \simeq H^{k+1}(F_{l-1}), 1 \leq k \leq n-2. $$
Since ${\rm hd}(F_{l-1}) = l-1$ by Proposition \ref{bsprop}, $H^{p}_{*}(F_{l-1}) = 0, 1 \leq p \leq n-l$ and there is
$t_0 \in \mathbb{Z}$ such that $H^{n-l+1}(F_{l-1}(t_0)) \neq 0$. Therefore we have
$$ H^{p}_{*}(F_l) = 0, 1 \leq p \leq n-l-1 $$
hence ${\rm hd}(F_l) \leq l$ by Proposition \ref{bsprop}. Since $H^{n-l}(F_l(t_0)) \simeq H^{n-l+1}(F_{l-1}(t_0)) \neq 0$, we cannot have ${\rm hd}(F_l) \leq l-1$, otherwise
$$ H^{p}_{*}(F_l) = 0, 1 \leq p \leq n-l. $$
Therefore ${\rm hd}(F_l) = l$ and we are done.
\end{proof}

\begin{lemma}\label{lalpha}
For each $1 \leq i \leq n-1$, $h^0(F_i^{*}(-t-n+i)) = \alpha_{n-i}$.
\end{lemma}

\begin{proof}
Since the complex $(\ref{compG})$ is self-dual up to twist (see \cite[Theorem 1.1.27]{MR}) we only need to check our claim for
$1 \leq i \leq \frac{n-1}{2}$. Proving the Lemma is equivalent to show that
$${\rm h}^0(F^{*}_{n-j}(-t-j)) = \alpha_j, \; {\rm for}\; 1 \leq j\leq \frac{n-1}{2}.$$
Taking the duals in $(\ref{seqsimp})$ we have

\begin{eqnarray}\label{seqdual}
\xymatrix{0 \ar[r] & F_1^{*} \ar[r] & \opn(t+n)^{\alpha_n} \ar[r] & \opn(2t+n+1)\ar[r] & 0 }
\end{eqnarray}
\begin{eqnarray*}\xymatrix{ 0 \ar[r] & F_2^{*} \ar[r] & \opn(t+n-1)^{\alpha_{n-1}} \ar[r] & F_{1}^{*} \ar[r] & 0} \end{eqnarray*}
\begin{eqnarray*}\xymatrix{ & & \vdots & & }\end{eqnarray*}
\begin{eqnarray*}\xymatrix{ 0 \ar[r] & F^{*}_{n-p+1} \ar[r] & \opn(t+p)^{\alpha_p} \ar[r] & F_{n-p}^{*} \ar[r] & 0}\end{eqnarray*}
\begin{eqnarray*}\xymatrix{& & \vdots & & }\end{eqnarray*}
\begin{eqnarray*}\xymatrix{  0 \ar[r] & F^{*}_{n-1}\ar[r] & \opn(t+2)^{\alpha_2} \ar[r] & F^{*}_{n-2} \ar[r] & 0}\end{eqnarray*}
\begin{eqnarray*}\xymatrix{  0 \ar[r] & \opn \ar[r] & \opn(t+1)^{\alpha_1} \ar[r] & F^{*}_{n-1} \ar[r] & 0 }\end{eqnarray*}

Proceeding by induction on $j$, first consider the case $j=1$:  twisting $(\ref{seqdual})$ with $\opn(-t-1)$ we get
$$\dim{\rm H}^0(F_{n-1}^{*}(-t-1)) = \alpha_{1}.$$

Next, suppose the Lemma is true for $2 \leq l \leq n-2$. Twist $(\ref{seqdual})$ with $\opn(-t-l-1)$, then we have

$$ \xymatrix{0 \ar[r] & F^{*}_{n-l}(-t-l-1) \ar[r] & \opn^{\alpha_{l+1}} \ar[r] & F^{*}_{n-l-1}(-t-l-1) \ar[r] & 0 }. $$
From the long exact sequence of cohomologies

$$ \xymatrix{0 \ar[r] & H^{0}( F^{*}_{n-l}(-t-l-1)) \ar[r] & H^{0}( \opn)^{\alpha_{l+1}} \ar[r] & \\
\ar[r] & H^{0}(F^{*}_{n-l-1}(-t-l-1)) \ar[r] & H^{1}(F^{*}_{n-l}(-t-l-1)) \ar[r] & 0 } $$

By Proposition \ref{bsprop}, $H^{1}(F^{*}_{n-l}(-t-l-1)) = 0$. Now we want to check that $H^{0}(F^{*}_{n-l}(-t-l-1)) = 0$. Twisting $(\ref{seqdual})$ with $\opn(-t-l-1)$ we have

$$\xymatrix{0 \ar[r] & F^{*}_{n-l+1}(-t-l-1) \ar[r] & \opn(-1)^{\alpha_l} \ar[r] & F^{*}_{n-l}(-t-l-1) \ar[r] & 0}.$$ Using Proposition \ref{bsprop}, $H^{0}(F^{*}_{n-l}(-t-l-1)) = 0$ and therefore $\dim H^{0}(F^{*}_{n-l-1}(-t-l-1))) = \alpha_{l+1}.$
\end{proof}

We are finally ready to establish the main result of this Section.

\begin{theorem}\label{teoGor}
The syzygies $F_i$ are simple for $i=1, \cdots, n$.
\end{theorem}
\begin{proof}
We know that the resolution is self dual up to twist, \cite[Prop. 1.1.27]{MR},  thus we only need to check for $F_i$, $1 \leq i \leq \frac{n-1}{2}$. Let us prove by induction on $i$. We can not apply Theorem \ref{decon} because the map
$\beta : \opn(-t-1)^{\alpha_1} \rightarrow \opn$ is not generic. Thus we need an ad hoc proof.

Since the complex (\ref{compG}) is self dual, we have
$$ F_{n-1} \simeq F_1^{*}(-2t-n-1). $$
Thus it is sufficient to prove that $F_{n-1}$ is simple.

Applying ${\rm Hom}( -,F_{n-1})$ to the sequence
$$\xymatrix{0 \ar[r] & F_{n-1} \ar[r] & \opn(-t-1)^{\alpha_1} \ar[r] & \opn \ar[r] & 0} $$
it is enough to check
$$ {\rm Hom}(\opn(-t-1)^{\alpha_1},F_{n-1}) = H^{0}(F_{n-1}(t+1))^{\alpha_1} = 0. $$ We have

$$ \xymatrix{0 \ar[r] & {\rm Hom}(F_{n-1},F_{n-1}) \ar[r] & k }. $$ Since
$1\in{\rm Hom}(F_{n-1},F_{n-1})$ then Hom$(F_{n-1},F_{n-1}) \simeq k$ and $F_{n-1}$ is simple.
Twisting all sequences of $(\ref{seqsimp})$ by $\opn(t+1)$ and applying the long exact sequence of cohomologies, one concludes that
$$ H^{0}(F_{n-1}(t+1)) \simeq H^{1}(F_{n-2}(t+1)) \simeq H^{j}(F_{n-j-1}(t+1)) \; \mbox{for}\; 1 \leq j \leq n-2. $$
However,
$$ H^{n-2}(F_{1}(t+1)) \simeq H^{n-1}(\opn(-t-n))=0. $$
Thus $H^{0}(F_{n-1}(t+1))^{a_1} = 0$ therefore $F_{n-1}$ is simple.


Now that we know that $F_1$ is simple, let us prove by induction on $i$ that $F_i$ is simple for $2 \leq i \leq \frac{n-1}{2}$. Suppose
$F_{k-1}$ is simple. Let $\gamma_{k}$ be a generic map

$$\gamma_k: F_{k-1} \rightarrow \opn(-t-(n-k+1))^{\alpha_{n-k+1}}.$$ We have

\begin{itemize}

\item[$(1)$] $F_{k-1}$ and $\opn(-t-n+k-1)$ are simple;

\item[$(2)$] Hom$(\opn(-t-n+k-1),F_{k-1}) \simeq H^{0}(F_{k-1}(t+n-k+1)) = 0$

\end{itemize}

In fact, tensoring $(\ref{seqsimp})$ by $\opn(t+n-k+1)$ we have the sequence

$$\xymatrix{0 \ar[r] & F_{k-2}(t+n-k+1) \ar[r] & \opn(-1)^{\alpha_{n-k+2}} \ar[r] & F_{k-1}(t+n-k+1) \ar[r] & 0}.$$ Therefore

$$H^{0}(F_{k-1}(t+n-k+1)) \simeq H^{1}(F_{k-2}(t+n-k+1)).$$ By Lemma \ref{dimhomcg}, ${\rm hd}(F_{k-2}) = k-2$, thus by Proposition \ref{bsprop}, $H^1(F_{k-2}(t+n-k+1)) = 0$.

\begin{itemize}

\item[$(3)$] ${\rm Ext}^{1}(\opn(-t-n+k-1), F_{k-1}) \simeq H^{1}(F_{k-1}(t+n-k+1)) = 0 $

\end{itemize}

 Indeed, since we have hd$(F_{k-1}) = k-1$, then $H^{1}(F_{k-1}(t+n-k-1)) = 0$ by Proposition \ref{bsprop}.

\begin{itemize}

\item[$(4)$] $F^{*}_{k-1} \otimes \opn(-t-n+k-1) \simeq F^{*}_{k-1}(-t-n+k-1)$ is globally generated.

\end{itemize}

This is true because from the complex (\ref{seqsimp}) we have
$$ F^{*}_{j} \simeq F_{n-j}(2t+n+1). $$
Twisting $(\ref{seqsimp})$ by $\opn(t+k)$ we see that $F_{n-k+1}(t+k)$ is globally generated.

\begin{itemize}

\item[$(5)$] $\dim {\rm Hom}(F_{k-1}, \opn(-t-n+k-1)) \simeq h^{0}(F^{*}_{k-1}(-t-n+k-1)) = \alpha_{n-k+1}$, by Lemma $\ref{lalpha}$.
\end{itemize}

Therefore $\overline{F}_{k} = {\rm coker}\alpha_k$ is a cokernel bundle, by definition. Since

$$
q(1, \alpha_{n-k+1}) = 1 + \alpha_{n-k+1}^2 - \alpha_{n-k+1}(h^{0}(F^{*}_{k-1}(-t-n+k-1))) = 1, $$
it follows from Theorem $\ref{decon}$ that $\overline{F}_{k}$ is simple.





Since $\overline{F}^{*}_k(-2t-n-1)$ and $F_{n-k}$ are both the $k-$syzygies of the map $\beta$, they must be isomorphic. Therefore $F_{n-k}$ is simple and we are done.
\end{proof}

The results proved in Theorem \ref{Kos} and Theorem \ref{teoGor} for the syzygies of the pure resolutions (\ref{Kcomp}) and
(\ref{compG}), respectively, point to an interesting possible generalization. More precisely, given a pure resolution of the form
\begin{eqnarray} \label{pureresol}
\xymatrix{0 \ar[r] & \opn(-d_p)^{\beta_{p}} \ar[r] & \cdots \ar[r] & \opn(-d_1)^{\beta_1} \ar[r] & \opn \ar[r] & 0},
 \end{eqnarray}
for which values of the integers $d_1 < \cdots < d_p$ and $\beta_i$, $i=1, \cdots, p$, $p \geq 2$, are its syzygies simple vector bundles?


\end{document}